\documentclass[12pt, reqno, psamsfonts]{amsart}
\pdfoutput=1

\usepackage{amssymb}
\usepackage{amsthm}
\usepackage{amsmath}
\usepackage{latexsym}
\usepackage[T1]{fontenc}
\usepackage[utf8]{inputenc}
\usepackage[russian, french, english]{babel}

\usepackage{graphicx}
\usepackage{wrapfig}
\usepackage[justification=centering, labelfont=bf]{caption}
\usepackage{mathtools}
\usepackage{amsbsy}
\usepackage[inline]{enumitem}
\usepackage{array}
\usepackage{multicol}
\usepackage{cancel}

\usepackage{fouriernc}
\usepackage{framed}

\usepackage{upgreek}
\usepackage{titlesec}
\usepackage[spacing=true,kerning=true,babel=true,tracking=true]{microtype}
\usepackage[shortcuts]{extdash}
\usepackage[foot]{amsaddr}
\usepackage[left=1.2in,right=1.2in,top=1.2in,bottom=1.2in,bindingoffset=0cm]{geometry}
\usepackage{bm}
\usepackage{centernot}
\usepackage{mdframed}
\usepackage[hidelinks]{hyperref}
\usepackage{xspace}

\usepackage{tikz}
\usetikzlibrary{shapes,snakes}
\usetikzlibrary{arrows.meta}
\usetikzlibrary{decorations.pathmorphing}
\usetikzlibrary{patterns}
\usetikzlibrary{hobby}
\usetikzlibrary{decorations.markings}
\usepackage[justification=centering, labelfont=bf]{caption}
\usepackage{xspace}
\usepackage{float}

\usepackage[backend=biber, style=alphabetic, sorting=nyt, maxnames=100,backref=true]{biblatex}
\title{Descriptive Combinatorics and Distributed Algorithms}
\date{}

\author{Anton~Bernshteyn}
\address{\normalfont School of Mathematics, Georgia Institute of Technology, Atlanta, GA, USA}
\email{bahtoh@gatech.edu}
\thanks{This is a draft of an article to appear in the October 2022 issue of the \emph{Notices of the AMS}. This work is partially supported by the NSF grant DMS-2045412.}

\newtheoremstyle{bfnote}%
{}{}%
{\slshape}{}%
{\bfseries}{\bfseries.}%
{ }%
{\thmname{#1}\thmnumber{ #2}\thmnote{ \ep{\normalfont{}#3}}}

\theoremstyle{bfnote}
\newtheorem{theo}{Theorem}[section]
\newtheorem*{theo*}{Theorem}
\newtheorem{prop}[theo]{Proposition}

\newtheorem*{corl*}{Corollary}
\newtheorem{axiom}[theo]{Axiom}
\newtheorem*{claim*}{Claim}

\theoremstyle{definition}
\newtheorem{defn}[theo]{Definition}
\newtheorem*{defn*}{Definition}

\newtheorem{ques}[theo]{Problem}

\newtheorem*{exmp*}{Example}

\theoremstyle{remark}
\newtheorem*{ques*}{Question}
\newtheorem*{remk*}{Remark}

\newcommand*{\myproofname}{Proof}
\newenvironment{claimproof}[1][\myproofname]{\begin{proof}[#1]}{\end{proof}}

\newcommand{\0}{\varnothing}
\newcommand{\set}[1]{\{#1\}}
\newcommand{\N}{{\mathbb{N}}}
\newcommand{\Z}{\mathbb{Z}}
\newcommand{\F}{\mathbb{F}}
\newcommand{\R}{\mathbb{R}}

\newcommand{\Q}{\mathbb{Q}}

\renewcommand{\epsilon}{\varepsilon}

\renewcommand{\phi}{\varphi}
\renewcommand{\theta}{\vartheta}
\renewcommand{\leq}{\leqslant}
\renewcommand{\geq}{\geqslant}
\newcommand{\defeq}{\coloneqq}

\newcommand{\bemph}[1]{{\normalfont#1}} 
\newcommand{\ep}[1]{\bemph{(}#1\bemph{)}} 

\newcommand{\red}[1]{{\color{red}#1}}

\newcommand{\symdif}{\vartriangle}

\newcommand{\LOCAL}{$\mathsf{LOCAL}$\xspace}

\numberwithin{equation}{section}

\titleformat{\section}[block]{\large\bfseries}{\thesection.}{1ex}{}
\titleformat{\subsection}[block]{\bfseries}{\thesubsection.}{1ex}{}
\titleformat{\subsubsection}[block]{\itshape}{\bfseries\upshape\thesubsubsection.}{1ex}{}

\titlespacing*{\section}{0pt}{*3}{*1}
\titlespacing*{\subsection}{0pt}{*3}{*1}
\titlespacing*{\subsubsection}{0pt}{*1.5}{*1}

\renewbibmacro{in:}{}

\renewbibmacro*{volume+number+eid}{%
	\printfield{volume}%
	\setunit*{\addnbspace}
	\printfield{number}%
	\setunit{\addcomma\space}%
	\printfield{eid}}

\DeclareFieldFormat[article]{volume}{\textbf{#1}\space}
\DeclareFieldFormat[article]{number}{\mkbibparens{#1}}

\DeclareFieldFormat{journaltitle}{#1,}
\DeclareFieldFormat[thesis]{title}{\mkbibemph{#1}\addperiod}
\DeclareFieldFormat[article, unpublished, thesis]{title}{\mkbibemph{#1},}
\DeclareFieldFormat[book]{title}{\mkbibemph{#1}\addperiod}
\DeclareFieldFormat[unpublished]{howpublished}{#1, }

\DeclareFieldFormat{pages}{#1}

\DeclareFieldFormat[article]{series}{Ser.~#1\addcomma}

\setlist{topsep=3pt,itemsep=3pt}

\begin{document}

\maketitle

    \section*{}
    In this article we shall explore a fascinating area called \emph{descriptive combinatorics} and its recently discovered connections to \emph{distributed algorithms}---a fundamental part of computer science that is becoming increasingly important in the modern era of decentralized computation. The interdisciplinary nature of these connections means that there is very little common background shared by the researchers who are interested in them. With this in mind, this article was written under the assumption that the reader would have close to no background in either descriptive set theory \emph{or} computer science. The reader will judge to what degree this endeavor was successful.
    
    The article comprises two parts. In the first part we give a brief introduction to some of the central notions and problems of descriptive combinatorics. The second part is devoted to a survey of some of the results concerning the interactions between descriptive combinatorics and distributed algorithms, as well as a few open problems.
    
    
    \section{A brief introduction to descriptive combinatorics}
    
    \subsection{Basic notions of descriptive set theory}
    
    \subsubsection{Countable computation and Borel sets}
    
    Descriptive combinatorics is an area that emerged quite recently \ep{about two decades ago} as the result of a symbiosis between \emph{combinatorics} \ep{especially graph theory} and \emph{descriptive set theory}. For excellent surveys of this subject, see \cite{KechrisMarks} by Kechris and Marks and \cite{Pikh_survey} by Pikhurko. In addition to combinatorics and descriptive set theory, descriptive combinatorics has close connections to numerous other branches of mathematics, such as ergodic theory, topological dynamics, probability theory, and model theory, to name a few.
    
    To motivate the questions studied in descriptive combinatorics, it will be beneficial to first briefly discuss descriptive set theory more abstractly.
    
    \begin{figure}[t]

    \begin{center}    

    \begin{tikzpicture}[xscale=1.5]
        \node at (-1.2,0) {inputs};
    
        \node[circle,draw] (x0) at (0,0) {$x_0$};
        \node[circle, draw] (x1) at (1,0) {$x_1$};
        \node[circle, draw] (x2) at (2,0) {$x_2$};
        \node[circle, draw] (x3) at (3,0) {$x_3$};
        \node[circle, draw] (x4) at (4,0) {$x_4$};
        \node[circle, draw] (x5) at (5,0) {$x_5$};
        \node[circle,inner sep=5] (dot0) at (6,0) {$\ldots$};
        
        \node[circle, draw,inner sep=2] (and0) at (0,2) {$\mathtt{and}$};
        \node[circle, draw,inner sep=2] (and1) at (1,2) {$\mathtt{and}$};
        \node[circle, draw,inner sep=2] (and2) at (2,2) {$\mathtt{and}$};
        \node[circle, draw,inner sep=2] (and3) at (3,2) {$\mathtt{and}$};
        \node[circle, draw,inner sep=2] (and4) at (4,2) {$\mathtt{and}$};
        \node[circle, draw,inner sep=2] (and5) at (5,2) {$\mathtt{and}$};
        \node[circle,inner sep=5] (dot1) at (6,2) {$\ldots$};
        
        \node[circle,draw] (or) at (2.5,4) {$\mathtt{or}$};
        
        \node at (1.3,6) {output};
        \node[circle,draw,inner sep=2] (not) at (2.5,6) {$\mathtt{not}$};
        
        \begin{scope}[thick,decoration={markings,mark=at position 0.5 with {\arrow{stealth}}}] 
        \draw[postaction={decorate}] (x0) -- (and0);
        \draw[postaction={decorate}] (x1) -- (and0);
        
        \draw[postaction={decorate}] (x1) -- (and1);
        \draw[postaction={decorate}] (x2) -- (and1);
        
        \draw[postaction={decorate}] (x2) -- (and2);
        \draw[postaction={decorate}] (x3) -- (and2);
        
        \draw[postaction={decorate}] (x3) -- (and3);
        \draw[postaction={decorate}] (x4) -- (and3);
        
        \draw[postaction={decorate}] (x4) -- (and4);
        \draw[postaction={decorate}] (x5) -- (and4);
        
        \draw[postaction={decorate}] (x5) -- (and5);
        \draw[postaction={decorate}] (dot0) -- (and5);
        
        \draw[postaction={decorate}] (dot0) -- (dot1);
        
        \draw[postaction={decorate}] (and0) -- (or);
        \draw[postaction={decorate}] (and1) -- (or);
        \draw[postaction={decorate}] (and2) -- (or);
        \draw[postaction={decorate}] (and3) -- (or);
        \draw[postaction={decorate}] (and4) -- (or);
        \draw[postaction={decorate}] (and5) -- (or);
        \draw[postaction={decorate}] (dot1) -- (or);
        
        \draw[postaction={decorate}] (or) -- (not);
        \end{scope}
    \end{tikzpicture}

    \caption{A countable circuit that outputs $1$ if the input bit string does not contain two consecutive $1$s.\label{fig:infcirc}}
    \end{center}
    \end{figure}
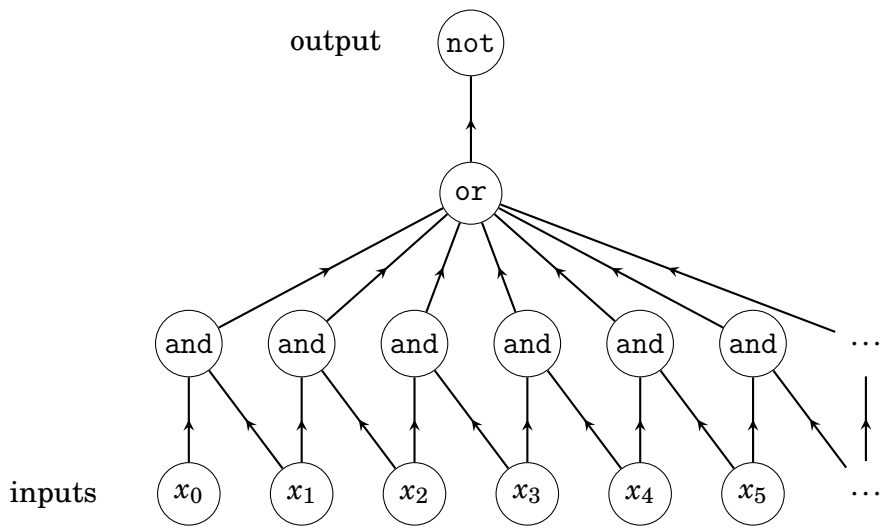
    
    One way \ep{out of many} of thinking about descriptive set theory is that it provides a versatile framework for gauging the inherent difficulty of mathematical problems pertaining to countable structures, in a manner analogous to how computational complexity theory gauges the inherent difficulty of finite problems. To pursue this analogy further, we can naturally present the basic concepts of descriptive set theory using a particular model of computation, namely \emph{Boolean circuits}. An example of an \ep{infinite} Boolean circuit is shown in Fig.~\ref{fig:infcirc}. It is a network consisting of nodes, also called \emph{gates}, joined by directed edges, or \emph{wires}. The bottom layer comprises the \emph{input nodes}. Each input node receives a bit value---$0$ or $1$. The values then propagate through the network along the wires, with every gate computing a particular Boolean function of the values that feed into it: \texttt{not}, \texttt{and}, or \texttt{or}. Finally, one or more nodes are designated as the \emph{outputs}. Thus, a Boolean circuit can be used to compute a function $\set{0,1}^{\mathrm{In}} \to \set{0,1}^{\mathrm{Out}}$, where $\mathrm{In}$ and $\mathrm{Out}$ are the sets of the input and the output nodes respectively. For example, the circuit in Fig.~\ref{fig:infcirc} computes the function $\set{0,1}^\N \to \set{0,1}$ that equals $1$ if and only if the input string does not contain two consecutive $1$s.
    
    There is one technical issue that we must make explicit here. Not every directed network can be used to perform well-defined computation. For instance, if a network involves a directed cycle of nodes, such as $v_0 \leftarrow v_1 \leftarrow v_2 \leftarrow \cdots \leftarrow v_0$, then the computational process wouldn't even be able to start. More generally, a Boolean circuit must be \emph{well-founded}, meaning that it must not contain an infinite descending sequence of nodes such as $v_0 \leftarrow v_1 \leftarrow v_2 \leftarrow v_3 \leftarrow \cdots$. It is not hard to show that well-foundedness is the only necessary requirement: any well-founded network of gates implements a well-defined function.
    
    An important part of computational complexity theory is \emph{circuit complexity}, which studies how large a \ep{finite} circuit needs to be to compute a given function $f \colon \set{0,1}^n \to \set{0,1}^m$. In descriptive set theory we are similarly interested in functions that can be computed by \emph{countable} circuits:
    
    \begin{defn}[Borel subsets of $\set{0,1}^\N$]\label{defn:Borel}
        A subset $A \subseteq \set{0,1}^\N$ is \emph{Borel} if its characteristic function can be computed by a countable Boolean circuit. We let $\mathfrak{B}(\set{0,1}^\N)$ denote the family of all Borel subsets of $\set{0,1}^\N$.
    \end{defn}
    
    While Definition~\ref{defn:Borel} applies to subsets of $\set{0,1}^\N$, it can naturally be extended to sets of other types, as long as their members can be somehow encoded by infinite bit strings. For example, a countable graph $G$ with vertex set $V = \set{v_0, v_1, v_2, \ldots}$ can be represented by its \emph{adjacency matrix} $M_G$, i.e., a countable table of $0$s and $1$s whose entry in row $i$ and column $j$ is $1$ if and only if $v_i$ and $v_j$ are adjacent in $G$. We can then call a set $A$ of countable graphs \emph{Borel} if there is a countable Boolean circuit that decides, given the matrix $M_G$, whether $G$ is in $A$ or not. For instance, the following sets of countable graphs are Borel:
    \begin{align*}
        \set{G \,:\, \text{$G$ is connected}}, \qquad \set{G \,:\, \text{$G$ is bipartite}}, \qquad \set{G \,:\, \text{$G$ is $3$-colorable}}.
    \end{align*}
    \ep{The last one may be a bit surprising, since deciding whether a finite graph is $3$-colorable is believed to be a computationally difficult problem. The difficulty miraculously disappears in the countable world, however.} On the other hand, the following set can be shown to \emph{not} be Borel:
    \[
        \set{G \,:\, \text{$G$ has an infinite clique}}.
    \]
    In fact, this set is \emph{complete analytic}, which is a notion somewhat analogous to \textsf{NP}-completeness from computational complexity, with the added benefit that in descriptive set theory it is a theorem \ep{due to Suslin from 1917; see \cite[\S26]{KechrisDST}} that complete analytic sets cannot be Borel. In general, \emph{most} subsets of $\set{0,1}^\N$ are not Borel: $\set{0,1}^\N$ has $2^{2^{\aleph_0}}$ subsets but only $2^{\aleph_0}$ of them are Borel \ep{this result follows by enumerating all countable Boolean circuits}.
    
    \subsubsection{Borel sets in Polish spaces}\label{subsec:polish}
    
    A useful \ep{and more classical} perspective on Borel sets is provided by topological considerations. We can define the distance between two infinite bit strings $x$, $y \in \set{0,1}^\N$ by the formula
    \[
        \mathrm{dist}(x,y) \,\defeq\, \sum_{n = 0}^\infty \frac{1}{2^{n+1}} |x_n - y_n|.
    \]
    (This can be thought of as a weighted version of the Hamming distance.) The coefficients $1/2^{n+1}$ are chosen so that the series $\sum_{n=0}^\infty 1/2^{n+1}$ converges \ep{and hence the metric $\mathrm{dist}$ is bounded, in this case by $1$}, but otherwise they are arbitrary. In particular, using a different convergent series with positive terms would yield a different metric, but the resulting topology on the space $\set{0,1}^\N$ would be the same---namely, it is the product topology arising from viewing $\set{0,1}^\N$ as the product of countably many copies of the discrete space $\set{0,1}$. Actually, it is not hard to explicitly describe this topology without referring to the metric: a set $A \subseteq \set{0,1}^\N$ is open if and only if for every sequence $x = (x_0, x_1, x_2, \ldots) \in A$, the membership of $x$ in $A$ can be confirmed by looking at only finitely many entries of $x$. More formally, $A$ is open if for each $x \in A$, there is some $n \in \N$ such that every infinite bit string whose first $n$ entries are $(x_0, x_1, \ldots, x_{n-1})$ is in $A$.
    
    Now, it turns out that Borel subsets of $\set{0,1}^\N$ can be described in purely topological terms and without any reference to Boolean circuits as follows: $\mathfrak{B}(\set{0,1}^\N)$ is the smallest family of subsets of $\set{0,1}^\N$ that includes all open sets and is closed under countable Boolean operations \ep{i.e., complements, countable unions, and countable intersections}. And this characterization can be taken as the definition of Borel subsets of any topological space:
    
    \begin{defn}[Borel sets in topological spaces]\label{defn:Borel_general}
        Let $X$ be a topological space. We let $\mathfrak{B}(X)$ be the smallest family of subsets of $X$ that includes all open sets and is closed under countable Boolean operations. A subset $A \subseteq X$ is \emph{Borel} if it is a member of $\mathfrak{B}(X)$.
    \end{defn}
    
    Although Definition~\ref{defn:Borel_general} makes sense for an arbitrary topological space $X$, descriptive set theory is mostly concerned with so-called Polish spaces. Formally, a topological space is \emph{Polish} if it is second-countable \ep{i.e., it has a countable basis} and completely metrizable \ep{i.e., the topology is generated by a complete metric}. A particularly compelling reason for focusing on Polish spaces is that the points of a Polish space can be encoded by infinite bit strings in a ``well-behaved'' way. This statement is made precise by the following remarkable theorem:
    
    \begin{theo}[{Borel Isomorphism Theorem \cite[Theorem 15.6]{KechrisDST}}]
        If $X$ is an uncountable Polish space, then there is a bijection $\mathbf{code} \colon X \to \set{0,1}^\N$ such that a set $A \subseteq X$ is Borel in $X$ if and only if the set $\set{\mathbf{code}(x) \,:\, x \in A}$ is Borel in $\set{0,1}^\N$.
    \end{theo}
    
    In other words, the computational definition of Borel subsets of $\set{0,1}^\N$ given in Definition~\ref{defn:Borel} can also be used to identify, via an appropriate coding, Borel sets in any uncountable Polish space. This is a powerful observation because many natural examples of topological spaces are Polish, for instance $\R$ and, more generally, $\R^n$ for $n \in \N$, the infinite-dimensional space $\R^\N$, the Baire space $\N^\N$, the unit intervals $[0,1]$, $(0,1)$, and $[0,1)$, the unit circle $\mathbb{S}^1$ and, more generally, any second-countable topological manifold, all compact metric spaces, all separable Banach spaces, etc. Furthermore, it is possible to parameterize various classes of mathematical structures by points in suitably defined Polish spaces. We have already seen how to use adjacency matrices to form a Polish space of countable graphs. In a similar vein, one can assemble, e.g., a Polish space of countable groups. There are also natural Polish spaces of continuous functions, measure-preserving transformations, separable Banach spaces, etc. It is even possible to define a Polish space of all Polish spaces!\footnote{For the interested reader, we sketch the construction of the Polish space of Polish spaces. First, it turns out that every Polish space is homeomorphic to a closed subset of $\R^\N$. Now, let $(U_n)_{n \in \N}$ be a countable basis for the topology of $\R^\N$. To each closed set $X \subseteq \R^\N$, we assign a code $\mathbf{code}_X \in \set{0,1}^\N$ as follows: $\mathbf{code}_X(n) = 1$ if and only if $X \cap U_n \neq \0$. Let $\mathcal{X} \subseteq \set{0,1}^\N$ be the set of all codes of closed subsets of $\R^\N$. One can show that $\mathcal{X}$ is Polish in the relative topology inherited from $\set{0,1}^\N$. Thus, all Polish spaces are parameterized by the points of the Polish space $\mathcal{X}$. 
    }
    
    \subsubsection{Other classes of sets in descriptive set theory}
    
    In addition to Borel sets, there are various other classes of sets that play an important role in descriptive set theory. In this section we briefly introduce three such classes.
    
    Let us temporarily return to the space $\set{0,1}^\N$ of infinite bit strings. Recall that a set $A \subseteq \set{0,1}^\N$ is Borel if the membership in $A$ can be decided by a countable Boolean circuit. Sometimes it makes sense to inquire whether a set is not just Borel but \emph{open}. As mentioned in \S\ref{subsec:polish}, a set $A \subseteq \set{0,1}^\N$ is open if for every sequence $x \in A$, the membership of $x$ in $A$ can be confirmed by looking at only finitely many bits of $x$. 
    Of course, we can also investigate open sets in any other Polish space.
    
    On the other hand, sometimes a set we are working with may not be Borel and yet be ``almost'' Borel in a certain suitable sense. For instance, a set $A \subseteq \set{0,1}^\N$ is \emph{measurable} if there is a countable Boolean circuit that correctly decides the membership in $A$ for a \emph{random} input point $x \in \set{0,1}^\N$. More precisely, $A$ is measurable if there is a countable Boolean circuit $C$ with the following property. Let $\mathbf{1}_A$ denote the characteristic function of $A$. Sample a sequence $x= (x_0, x_1, x_2, \ldots) \in \set{0,1}^\N$ randomly by making each $x_i$ be $0$ or $1$ independently with probability $1/2$. Then $C(x) = \mathbf{1}_A(x)$ with probability $1$; in other words, $C$ may fail to correctly determine $x$'s membership in $A$, but the probability of failure is $0$. We can similarly study measurable sets in any Polish space $X$, provided that it is equipped with a Borel probability measure $\mu$.\footnote{We should point out that the class of measurable sets varies with the choice of the measure $\mu$. For simplicity, we shall assume throughout this article that whenever we speak of measurable sets, it is with respect to some implicitly fixed probability measure.} In many applications, sets of measure $0$ may be safely ignored, and thus measurability is often a ``good enough'' substitute for Borelness.
    
    Finally, we also often consider the so-called {Baire-measurable} sets. The notion of Baire-measurability is analogous to measurability, but it is typically easier to work with. Furthermore, it is defined purely topologically, without reference to a measure or any other additional structure. The role of sets of measure $0$ is played by the so-called \emph{meager} sets, i.e., countable unions of nowhere dense sets. We think of meager sets as ``topologically negligible.'' A set $A$ is \emph{Baire-measurable} if there is a Borel set $A'$ such that the symmetric difference $A \symdif A'$ is meager---informally, $A$ and $A'$ are ``topologically almost equal.'' The area studying meager and Baire-measurable sets is called \emph{Baire category theory} \ep{owing to the older term ``sets of first category'' for meager sets}. For a detailed comparison of measure and Baire category, see the excellent book \cite{Oxtoby} by Oxtoby.
    
    \subsection{Borel graphs and their combinatorics}\label{subsec:Borelgraphs}
    
    As mentioned earlier, descriptive combinatorics investigates classical combinatorial problems---such as graph coloring, matching, etc.---from the perspective of descriptive set theory. The central notion here is that of a Borel graph:
    
    \begin{defn}[Borel graphs]
        A \emph{Borel graph} is a graph $G$ whose vertex set $V(G)$ is a Polish space and whose edge set $E(G)$ is a Borel subset of $V(G) \times V(G)$.
    \end{defn}
    
    In this section, we describe some examples of Borel graphs and highlight a few of their properties that are of interest in descriptive combinatorics.
    
    \subsubsection{The hypercube}\label{subsec:cube}
    
    \begin{figure}[t]

    \begin{center}    
    
    \begin{tikzpicture}
    
    \begin{scope}[xshift=1cm]
        \filldraw (0,0.5) circle (1pt);
        \filldraw (0,-0.5) circle (1pt);
        \draw (0,-0.5) -- (0,0.5);
        \node[anchor=north] at (0,-0.5) {$0$};
        \node[anchor=south] at (0,0.5) {$1$};
    \end{scope}
    
    \begin{scope}[xshift=4cm]
        \filldraw (-0.5,0.5) circle (1pt);
        \filldraw (-0.5,-0.5) circle (1pt);
        \draw (-0.5,-0.5) -- (-0.5,0.5);
        \filldraw (0.5,0.5) circle (1pt);
        \filldraw (0.5,-0.5) circle (1pt);
        \draw (0.5,-0.5) -- (0.5,0.5);
        \draw (-0.5,-0.5) -- (0.5,-0.5);
        \draw (-0.5,0.5) -- (0.5,0.5);
        
        \node[anchor=north] at (-0.5,-0.5) {$00$};
        \node[anchor=north] at (0.5,-0.5) {$01$};
        \node[anchor=south] at (-0.5,0.5) {$10$};
        \node[anchor=south] at (0.5,0.5) {$11$};
    \end{scope}
    
    \begin{scope}[xshift=8cm,scale=1.5]
        \filldraw (-0.7,0.3) circle (1pt);
        \filldraw (-0.7,-0.7) circle (1pt);
        \draw (-0.7,-0.7) -- (-0.7,0.3);
        \filldraw (0.3,0.3) circle (1pt);
        \filldraw (0.3,-0.7) circle (1pt);
        \draw (0.3,-0.7) -- (0.3,0.3);
        \draw (-0.7,-0.7) -- (0.3,-0.7);
        \draw (-0.7,0.3) -- (0.3,0.3);
        
        \filldraw (-0.3,0.7) circle (1pt);
        \filldraw (-0.3,-0.3) circle (1pt);
        \draw (-0.3,-0.3) -- (-0.3,0.7);
        \filldraw (0.7,0.7) circle (1pt);
        \filldraw (0.7,-0.3) circle (1pt);
        \draw (0.7,-0.3) -- (0.7,0.7);
        \draw (-0.3,-0.3) -- (0.7,-0.3);
        \draw (-0.3,0.7) -- (0.7,0.7);
        
        \draw (-0.7,-0.7) -- (-0.3,-0.3);
        \draw (-0.7,0.3) -- (-0.3,0.7);
        \draw (0.3,-0.7) -- (0.7,-0.3);
        \draw (0.3,0.3) -- (0.7,0.7);
        
        \node[anchor=north] at (-0.7,-0.7) {$000$};
		\node[anchor=north] at (0.3,-0.7) {$010$};
		\node[anchor=east] at (-0.7,0.3) {$100$};
			\node[anchor=south] at (-0.3,0.7) {$101$};
			\node[anchor=180+40] at (-0.3,-0.3) {$001$};
			\node[anchor=west] at (0.7,-0.3) {$011$};
			\node[anchor=35] at (0.3,0.3) {$110$};
			\node[anchor=south] at (0.7,0.7) {$111$};
    \end{scope}
        
    \end{tikzpicture}

    \caption{$1$-, $2$-, and $3$-dimensional hypercubes. In contrast to the infinite-dimensional hypercube, these graphs are connected. \label{fig:hyper}}
    \end{center}
    \end{figure}
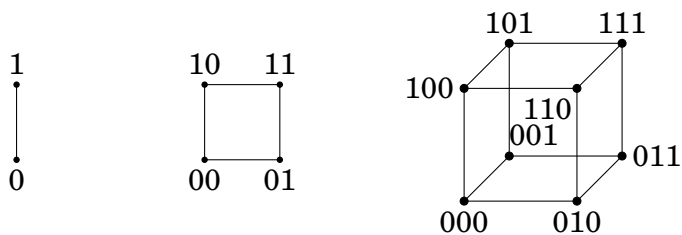
    
    For our first example, let $\mathbb{H}$ be the graph with vertex set $\set{0,1}^\N$ in which two bit strings $x = (x_0, x_1, x_2, \ldots)$, $y = (y_0, y_1, y_2, \ldots) \in \set{0,1}^\N$ are adjacent if and only if they differ in exactly one coordinate, i.e., if there is a unique index $n \in \N$ such that $x_n \neq y_n$. We call $\mathbb{H}$ the \emph{infinite-dimensional hypercube} \ep{it can be viewed as the inverse limit of finite-dimensional hypercubes, the first few of which are shown in Fig.~\ref{fig:hyper}}. The graph $\mathbb{H}$ is Borel, since it is easy to construct a countable Boolean circuit which, given two infinite bit strings $x$ and $y$, determines if they are adjacent in $\mathbb{H}$. 
    
    Note that the vertex set of $\mathbb{H}$ has cardinality $2^{\aleph_0}$. This is typical for graphs studied in descriptive combinatorics, since their vertex sets usually are uncountable Polish spaces.
    
    Next, we observe that, somewhat surprisingly and in contrast to finite-dimensional Boolean cubes, the graph $\mathbb{H}$ is disconnected. To see this, recall that the \emph{connected component} of a vertex $x$ in a graph $G$ is the smallest set $[x]_G \subseteq V(G)$ that contains $x$ and such that there are no edges joining $[x]_G$ to $V(G) \setminus [x]_G$. Equivalently, $[x]_G$ is the set of all vertices reachable from $x$ in $G$ by a {finite} path. A graph is {connected} if it has a single connected component. Now take any vertex $x \in \set{0,1}^\N$ in $\mathbb{H}$. 
        A moment's thought reveals that the connected component $[x]_\mathbb{H}$ is the set of all sequences $y \in \set{0,1}^\N$ that differ from $x$ in only finitely many coordinates. For instance, if $x = (0,0,0,\ldots)$, then $[x]_\mathbb{H}$ is the set of all bit strings with finitely many $1$s. It follows that every connected component of $\mathbb{H}$ is countable, which implies that $\mathbb{H}$ has $2^{\aleph_0}$-many components. This is also a typical feature of graphs studied in descriptive combinatorics.

    One of the most important parameters studied in graph theory is the chromatic number of a graph:
    
    \begin{defn}[Independent sets and chromatic numbers]
        Let $G$ be a graph. A subset $I \subseteq V(G)$ is \emph{$G$-independent} if no edge of $G$ joins two vertices in $I$. The \emph{chromatic number} $\chi(G)$ of $G$ is the smallest cardinal $\kappa$ such that $V(G)$ can be partitioned into $\kappa$-many $G$-independent sets.
    \end{defn}
    
    Let us compute the chromatic number of $\mathbb{H}$:
    
    \begin{prop}\label{prop:cube2}
        We have $\chi(\mathbb{H}) = 2$.
    \end{prop}
    \begin{proof}
        Clearly, $\chi(\mathbb{H}) \geq 2$. To prove that $\chi(\mathbb{H}) \leq 2$, we need to partition $\set{0,1}^\N$ into two $\mathbb{H}$-independent sets. To this end, choose an arbitrary representative from every connected component of $\mathbb{H}$. For $x \in \set{0,1}^\N$, let $\mathrm{rep}(x)$ be the representative from $[x]_\mathbb{H}$. Since $x$ and $\mathrm{rep}(x)$ belong to the same component of $\mathbb{H}$, they differ in finitely many coordinates. Thus, we can let $\delta(x)$ be the number of coordinates where $x$ and $\mathrm{rep}(x)$ differ and define 
        \begin{align*}
            A \,&\defeq\, \set{x \in \set{0,1}^\N \,:\, \text{$\delta(x)$ is even}},\\
            B \,&\defeq\, \set{x \in \set{0,1}^\N \,:\, \text{$\delta(x)$ is odd}}.
        \end{align*}
        If $x$ and $y$ are neighbors in $\mathbb{H}$, then $\mathrm{rep}(x) = \mathrm{rep}(y)$, and hence $|\delta(x) - \delta(y)| = 1$. Therefore, the sets $A$ and $B$ are $\mathbb{H}$-independent and partition $\set{0,1}^\N$, as desired.
    \end{proof}
    
    Since we know that $\chi(\mathbb{H}) = 2$, it makes sense to ask the following question:
    \begin{quote}
        \textsl{Can the set $\set{0,1}^\N$ be partitioned into two {Borel} $\mathbb{H}$-independent sets?}
    \end{quote}
    This is a special case of the following general problem:
    
    \begin{defn}[Borel chromatic numbers]
        Let $G$ be a Borel graph. The \emph{Borel chromatic number} $\chi_\mathsf{B}(G)$ of $G$ is the smallest $\kappa \in \set{0,1,2,\ldots, \aleph_0, 2^{\aleph_0}}$ such that $V(G)$ can be partitioned into $\kappa$-many Borel $G$-independent sets.\footnote{Note that, by definition, if $\chi_\mathsf{B}(G)$ is uncountable, then $\chi_\mathsf{B}(G) = 2^{\aleph_0}$. This convention is motivated by the fact that every uncountable Polish space has cardinality $2^{\aleph_0}$. However, it may still be meaningful to ask whether $V(G)$ can be partitioned into $\kappa$-many Borel $G$-independent sets for some $\aleph_0 < \kappa < 2^{\aleph_0}$; see \cite{Geschke} for further discussion. This issue will not play a major role in this article, since we shall be mostly interested in the case when $\chi_\mathsf{B}(G)$ is finite.} The \emph{measurable} and \emph{Baire-measurable chromatic numbers} $\chi_\mathsf{M}(G)$ and $\chi_\mathsf{BM}(G)$ are defined analogously.
    \end{defn}
    
    It turns out that, although the chromatic number of $\mathbb{H}$ is $2$, its Borel chromatic number is uncountable!
    
    \begin{prop}[{\cite[\S4.3]{KechrisMarks}}]\label{prop:cube}
        The set $\set{0,1}^\N$ cannot be partitioned into countably many Borel $\mathbb{H}$-independent sets. In other words, $\chi_{\mathsf{B}}(\mathbb{H}) = 2^{\aleph_0}$.
    \end{prop}
    
    Actually, we even have $\chi_{\mathsf{M}}(\mathbb{H}) = \chi_\mathsf{BM}(\mathbb{H}) = 2^{\aleph_0}$ \ep{and, of course, both $\chi_\mathsf{M}(\mathbb{H})$ and $\chi_\mathsf{BM}(\mathbb{H})$ are lower bounds for $\chi_\mathsf{B}(\mathbb{H})$}. To prove this, one shows that every measurable $\mathbb{H}$\=/independent set must have measure $0$, while every Baire-measurable $\mathbb{H}$-independent set must be meager. Since countable unions of measure-$0$ or meager sets are still measure-$0$ or meager respectively, it is impossible to cover the entire space $\set{0,1}^\N$ by countably many measurable or Baire-measurable $\mathbb{H}$-independent sets.
    
    One interpretation of Proposition~\ref{prop:cube} is that there is no ``explicit'' partition of the vertex set of $\mathbb{H}$ into two---or even countably many---$\mathbb{H}$-independent sets. What makes the construction in the proof of Proposition~\ref{prop:cube2} ``inexplicit'' is the very first step, namely choosing a single vertex in each connected component of $\mathbb{H}$. This step implicitly invokes the so-called \emph{Axiom of Choice}:
    
    \begin{axiom}[Choice]
        If $\mathcal{F}$ is a family of nonempty sets, then there is a way to pick one element from each set in $\mathcal{F}$. Formally, there is a function $\mathbf{choice}$ that assigns to each $A \in \mathcal{F}$ an element $\mathbf{choice}(A) \in A$.
    \end{axiom}
    
    The Axiom of Choice is part of the axiom system \textsf{ZFC} \ep{\emph{Zermelo--Fraenkel set theory with the Axiom of Choice}} and plays a crucial role in the foundations of mathematics. While most other axioms of \textsf{ZFC} assert the existence of concrete sets, such as the powerset $\mathcal{P}(A)$ of a given set $A$, the Axiom of Choice postulates the existence of a choice function without actually describing \emph{how} the choice is to be made. In this sense, the Axiom of Choice is non-constructive---so much so that in the past it was viewed with deep suspicion by many mathematicians. After all, how can we assert that something exists without being able to provide even a single example? 
     Although by now the controversy around the Axiom of Choice has largely died down and most mathematicians use it without reservation, it is still true that the Axiom of Choice---and hence our proof of Proposition~\ref{prop:cube2} based on it---is inherently non-constructive. From this perspective, Proposition~\ref{prop:cube} says that it is impossible to come up with an alternative, constructive argument---at least if by ``constructive'' we mean ``Borel.''\footnote{This is not the only way in which the Axiom of Choice affects chromatic numbers of graphs. For instance, the following remarkable fact was established by Galvin and Komj\'{a}th \cite{GK}: the statement that the chromatic number is well-defined for every graph $G$ is \emph{equivalent} to the Axiom of Choice.}
    
    \subsubsection{Translation and rotation}
    
    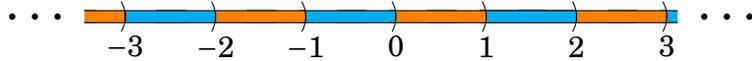
\begin{figure}[t]
		\centering	
		\begin{tikzpicture}[scale=0.8]
			
			

			\filldraw [fill=cyan] (8.9,-0.1) rectangle (10,0.1);
			\begin{scope}
			\clip (10,0) circle (0.5);
			\filldraw [fill=cyan] (9,-0.1) rectangle (10.5,0.1);
			\end{scope}
			
			\filldraw [fill=orange] (5.9,-0.1) rectangle (8.5,0.1);
			\begin{scope}
			\clip (8.5,0) circle (0.5);
			\filldraw [fill=orange] (6,-0.1) rectangle (9,0.1);
			\end{scope}
			
			\filldraw [fill=cyan] (4.4+1.5,-0.1) rectangle (5.5+1.5,0.1);
			\begin{scope}
			\clip (5.5+1.5,0) circle (0.5);
			\filldraw [fill=cyan] (4.5+1.5,-0.1) rectangle (6+1.5,0.1);
			\end{scope}
			
			\filldraw [fill=orange] (2.9+1.5,-0.1) rectangle (4+1.5,0.1);
			\begin{scope}
			\clip (4+1.5,0) circle (0.5);
			\filldraw [fill=orange] (3+1.5,-0.1) rectangle (4.5+1.5,0.1);
			\end{scope}
			
			\filldraw [fill=cyan] (1.4+1.5,-0.1) rectangle (2.5+1.5,0.1);
			\begin{scope}
			\clip (2.5+1.5,0) circle (0.5);
			\filldraw [fill=cyan] (1.5+1.5,-0.1) rectangle (3+1.5,0.1);
			\end{scope}
			
			\filldraw [fill=orange] (0+1.5,-0.1) rectangle (1+1.5,0.1);
			\begin{scope}
			\clip (1+1.5,0) circle (0.5);
			\filldraw [fill=orange] (0+1.5,-0.1) rectangle (1.5+1.5,0.1);
			\end{scope}
			
			\filldraw [fill=cyan] (0,-0.1) rectangle (1,0.1);
			\begin{scope}
			\clip (1,0) circle (0.5);
			\filldraw [fill=cyan] (0,-0.1) rectangle (1.5,0.1);
			\end{scope}
			
			\filldraw [fill=orange] (0-1.5,-0.1) rectangle (1-1.5,0.1);
			\begin{scope}
			\clip (1-1.5,0) circle (0.5);
			\filldraw [fill=orange] (0-1.5,-0.1) rectangle (1.5-1.5,0.1);
			\end{scope}
			
			\draw (1-1.5,0) ++(30:0.5) arc (30:-30:0.5);
			\draw (1,0) ++(30:0.5) arc (30:-30:0.5);
			\draw (2.5,0) ++(30:0.5) arc (30:-30:0.5);
			\draw (4,0) ++(30:0.5) arc (30:-30:0.5);
			\draw (5.5,0) ++(30:0.5) arc (30:-30:0.5);
			\draw (7,0) ++(30:0.5) arc (30:-30:0.5);
			\draw (8.5,0) ++(30:0.5) arc (30:-30:0.5);
			\draw (10,0) ++(30:0.5) arc (30:-30:0.5);
			
			\node[fill=white, shape=circle] at (0-1.5, 0) {\LARGE\,\ldots};
			\node[fill=white, shape=circle] at (10, 0) {\LARGE\,\ldots};
			
			\node at (0, -0.5) {$-3$};
			\node at (1.5, -0.5) {$-2$};
			\node at (3, -0.5) {$-1$};
			\node at (4.5, -0.5) {$0$};
			\node at (4.5+1.5, -0.5) {$1$};
			\node at (4.5+3, -0.5) {$2$};
			\node at (4.5+4.5, -0.5) {$3$};
			
		\end{tikzpicture}
		\caption{A partition of $\R$ into two Borel $G_{\mathrm{tr}}$-independent sets \ep{indicated by the colors}.}\label{fig:Borel2col}
	\end{figure}
    
    In the remainder of this article we shall be concerned with Borel graphs $G$ \emph{of bounded degree}, meaning that there is a natural number $d \in \N$ such that every vertex of $G$ has at most $d$ neighbors. A typical example is the graph $G_{\mathrm{tr}}$ with vertex set $\R$ in which vertices $x$ and $y$ are adjacent if and only if $|x-y| = 1$. In other words, each $x \in \R$ has precisely two neighbors in $G_{\mathrm{tr}}$: $x-1$ and $x+1$. It follows that every connected component of $G_{\mathrm{tr}}$ is a {bi-infinite path} \ep{i.e., a path infinite in both directions}. In particular, just like the hypercube graph $\mathbb{H}$ from \S\ref{subsec:cube}, $G_{\mathrm{tr}}$ has $2^{\aleph_0}$-many countable connected components. \ep{This fact can also be observed directly, since the vertices in the half-open interval $[0,1)$ belong to distinct components.} Clearly, $\chi(G_{\mathrm{tr}}) = 2$. In contrast to the graph $\mathbb{H}$ from \S\ref{subsec:cube} however, we can explicitly describe a partition of $G_{\mathrm{tr}}$ into two independent sets $A$ and $B$: for every integer $n \in \Z$, put the points in the interval $[n,n+1)$ in $A$ if $n$ is even and in $B$ if $n$ is odd \ep{{see Fig.~\ref{fig:Borel2col}}}. This yields
    \[
        A \,=\, \bigcup_{k \in \Z} [2k, 2k+1) \quad \text{and} \quad B \,=\, \bigcup_{k \in \Z} [2k-1, 2k).
    \]
    These sets are Borel, so we in fact have $\chi_\mathsf{B}(G_{\mathrm{tr}}) = 2$.
    
    \begin{figure}[t]
		\centering	
		\begin{tikzpicture}[scale=0.8]
			
			\draw (0,-0.25) -- (0, 0.25);
			
			\draw (0,0) -- (11,0);
			
			\filldraw [fill=olive] (10.4,-0.1) rectangle (10.5,0.1);
			\begin{scope}
			\clip (10.5,0) circle (0.5);
			\filldraw [fill=olive] (10.5,-0.1) rectangle (11.5,0.1);
			\end{scope}
			
			\filldraw [fill=cyan] (8.9,-0.1) rectangle (10,0.1);
			\begin{scope}
			\clip (10,0) circle (0.5);
			\filldraw [fill=cyan] (9,-0.1) rectangle (10.5,0.1);
			\end{scope}
			
			\filldraw [fill=orange] (5.9,-0.1) rectangle (8.5,0.1);
			\begin{scope}
			\clip (8.5,0) circle (0.5);
			\filldraw [fill=orange] (6,-0.1) rectangle (9,0.1);
			\end{scope}
			
			\filldraw [fill=cyan] (4.4,-0.1) rectangle (5.5,0.1);
			\begin{scope}
			\clip (5.5,0) circle (0.5);
			\filldraw [fill=cyan] (4.5,-0.1) rectangle (6,0.1);
			\end{scope}
			
			\filldraw [fill=orange] (2.9,-0.1) rectangle (4,0.1);
			\begin{scope}
			\clip (4,0) circle (0.5);
			\filldraw [fill=orange] (3,-0.1) rectangle (4.5,0.1);
			\end{scope}
			
			\filldraw [fill=cyan] (1.4,-0.1) rectangle (2.5,0.1);
			\begin{scope}
			\clip (2.5,0) circle (0.5);
			\filldraw [fill=cyan] (1.5,-0.1) rectangle (3,0.1);
			\end{scope}
			
			\filldraw [fill=orange] (0,-0.1) rectangle (1,0.1);
			\begin{scope}
			\clip (1,0) circle (0.5);
			\filldraw [fill=orange] (0,-0.1) rectangle (1.5,0.1);
			\end{scope}
			
			\draw (1,0) ++(30:0.5) arc (30:-30:0.5);
			\draw (2.5,0) ++(30:0.5) arc (30:-30:0.5);
			\draw (4,0) ++(30:0.5) arc (30:-30:0.5);
			\draw (5.5,0) ++(30:0.5) arc (30:-30:0.5);
			\draw (8.5,0) ++(30:0.5) arc (30:-30:0.5);
			\draw (10,0) ++(30:0.5) arc (30:-30:0.5);
			\draw (10.5,0) ++(30:0.5) arc (30:-30:0.5);
			
			\node[fill=white, shape=circle] at (7.5, 0) {\LARGE\,\ldots};
			
			\node at (0,-0.5) {$0$};
			\node at (11,-0.5) {$1$};
			\node at (1.5, -0.5) {$\alpha$};
			\node at (3, -0.5) {$2\alpha$};
			\node at (4.5, -0.5) {$3\alpha$};
			
		\end{tikzpicture}
		\caption{A partition of $[0,1)$ into three Borel $G_{\mathrm{rot}}$-independent sets \ep{indicated by the colors}.}\label{fig:Borel3col}
	\end{figure}
    
    Next we fix a number $\alpha \in (0,1) \setminus \Q$ and define a graph $G_{\mathrm{rot}}$ as follows. The vertex set of $G_{\mathrm{rot}}$ is the half-open unit interval $[0,1)$. Two vertices $x$, $y \in [0,1)$ are adjacent in $G_{\mathrm{rot}}$ if and only if $y - x = \pm \alpha \pmod 1$. Again, every vertex $x$ has precisely two neighbors, namely $x + \alpha \pmod 1$ and $x - \alpha \pmod 1$. \ep{Another way of thinking about this graph is by ``wrapping'' the unit interval around a circle, which turns adding $\alpha$ modulo $1$ into rotation by the angle $2 \pi\alpha$.} Since $\alpha$ is irrational, it is not hard to see that $G_{\mathrm{rot}}$ has $2^{\aleph_0}$-many connected components, each of which is a bi-infinite path. In other words, the graphs $G_{\mathrm{rot}}$ and $G_{\mathrm{tr}}$ are isomorphic. In particular, $\chi(G_{\mathrm{rot}}) = 2$. But can we describe a bipartition of $G_{\mathrm{rot}}$ explicitly? An attempt to adapt the construction used for $G_{\mathrm{tr}}$ fails, although it does yield a partition of $G_{\mathrm{rot}}$ into three Borel independent sets \ep{see Fig.~\ref{fig:Borel3col}}. In fact, it turns out that \emph{no explicit bipartition of $G_{\mathrm{rot}}$ exists}. This statement is made precise by the following proposition:
    
    \begin{prop}[{\cite[2]{CMTD}}]
        We have $\chi_\mathsf{B}(G_{\mathrm{rot}}) = 3$. Furthermore, $\chi_{\mathsf{M}}(G_{\mathrm{rot}}) = \chi_{\mathsf{BM}}(G_{\mathrm{rot}}) = 3$ as well.
    \end{prop}
    
    In other words, $G_{\mathrm{tr}}$ and $G_{\mathrm{rot}}$ are two graphs that are combinatorially isomorphic, but the topological structures on their vertex sets affect their Borel, measurable, and Baire-measurable chromatic numbers in different ways.
    
    \subsubsection{Locally checkable labeling problems}
    
    In the above examples, we have been interested in chromatic numbers of graphs as well as in their Borel, measurable, etc.~analogs. Naturally, there are many other combinatorial concepts we may want to study. Many of them fall into the general framework of \emph{locally checkable labelling \ep{LCL} problems} \ep{also known as \emph{local coloring problems}}. In an LCL problem $\Pi$, we are asked to assign labels from a finite set $\Lambda$ to the vertices of a graph $G$ in such a way that certain constraints are satisfied. The constraints must be ``local'' in the sense that they can be checked by looking at the labels of each vertex and its neighbors. For a formal \ep{and somewhat more general} definition, see, e.g., \cite[\S2.A.1]{BerDist}. For our purposes, it will suffice to give a few examples that illustrate what sort of problems fall into this category.
    
    The prototypical example is the \emph{$k$-coloring problem}: assign the labels $\Lambda = \set{1,\ldots,k}$, referred to as ``colors,'' to the vertices of $G$ so that the label of each vertex is distinct from the labels of its neighbors. Note that this is a ``local'' constraint in the sense described above: it only involves comparing the label of a vertex to those of its neighbors. By definition, the set of all vertices receiving a particular label in a $k$-coloring must be a $G$-independent set, and hence $G$ admits a $k$-coloring if and only if $\chi(G) \leq k$.
    
    Another well-studied problem is the \emph{Maximal Independent Set \ep{MIS} problem}: find an independent set $I \subseteq V(G)$ that is maximal under inclusion. This can be interpreted as an LCL problem as follows. We can encode a set $I \subseteq V(G)$ by its characteristic function, which assigns the labels $\Lambda = \set{0,1}$ to the vertices of $G$. It is easy to see that $I$ is an MIS if and only if this labelling satisfies the following ``local'' constraint:
    \begin{quote}
        \textsl{If a vertex is labeled $1$, then all its neighbors are labeled $0$, while if a vertex is labeled $0$, then at least one of its neighbors is labeled $1$.}
    \end{quote}
    
    As our last example, consider the \emph{perfect matching problem}: find a set of edges $M \subseteq E(G)$ such that every vertex is incident to exactly one edge in $M$. If $xy$ is an edge in $M$, then we say that the vertices $x$ and $y$ are \emph{matched} to each other. Thus, a perfect matching splits $V(G)$ into pairs of matched vertices. To place the perfect matching problem into the LCL framework, let us assume that $G$ is a graph of bounded degree; more precisely, suppose that every vertex of $G$ has at most $d$ neighbors, for some $d \in \N$. For each $x \in V(G)$, we fix an arbitrary ordering of the neighbors of $x$. \ep{If $G$ is a Borel graph, it is possible to fix such an ordering ``in a Borel way,'' in the sense that for each $i$, the set $\set{(x,y) \in V(G) \times V(G) \,:\, \text{$y$ is the $i$-th neighbor of $x$}}$ is Borel.} Then we can identify a perfect matching with a labeling of $V(G)$ using the labels $\Lambda = \set{1,\ldots,d}$ such that a vertex is given the label $i$ when it is matched to its $i$-th neighbor. Such labeling encodes a perfect matching if and only if the following constraint holds:
    \begin{quote}
        \textsl{Let $x$ be a vertex labeled $i$ and let $y$ be the $i$-th neighbor of $x$. If $y$'s label is $j$, then $x$ is the $j$-th neighbor of $y$.}
    \end{quote}
    In other words, if $x$ is matched to $y$, then $y$ must be matched to $x$. This constraint is again ``local,'' which makes the perfect matching problem an LCL problem \ep{at least on bounded degree graphs}.
    
    In descriptive combinatorics, we are interested in Borel solutions to LCL problems: 
    
    \begin{defn}[Borel solutions to LCL problems]\label{defn:BorelLCL}
        Let $\Pi$ be an LCL problem with label set $\Lambda$ and let $G$ be a Borel graph. A labeling of the vertices of $G$ by the labels from $\Lambda$ is a \emph{Borel solution} to $\Pi$ if  it fulfills all the constraints of the problem $\Pi$ and, additionally, for each label $\lambda \in \Lambda$, the set of all vertices labeled $\lambda$ is Borel.
        
        \emph{Measurable} and \emph{Baire-measurable solutions} to $\Pi$ are defined analogously.
    \end{defn}
    
    By applying Definition~\ref{defn:BorelLCL} to specific LCL problems, we obtain, as special cases, Borel $k$-colorings, Borel maximal independent sets, and Borel perfect matchings. In general, a lot of research in descriptive combinatorics can be seen as addressing the following comprehensive question:
    
    \begin{ques}[Descriptive complexity of LCL problems]\label{prob:general}
        Given a class $\mathcal{G}$ of Borel graphs, which LCL problems admit Borel/measurable/etc.~solutions on the graphs in $\mathcal{G}$?
    \end{ques}
    
    As stated, this question may seem way too general to say anything meaningful about. However, rather surprisingly, some recent work has led to significant progress toward a complete answer! The key to this new work is the discovery of an intimate connection between the answer to Question~\ref{prob:general} and the complexity of solving the given LCL problem on {finite} graphs using a distributed algorithm. The second part of this paper is dedicated to a survey of this connection.
    
    \section{Connections to distributed computing}
    
    \subsection{The \LOCAL model of distributed computation}
    
    The area of distributed computing studies computational processes performed by decentralized groups of independent agents. There are several different models of distributed computation that emphasize various aspects of it, such as fault-tolerance, communication bandwidth, etc. For our purposes, the relevant model is called \LOCAL and was formally introduced by Linial in 1992 \cite{Linial} \ep{although some related results had already been established somewhat earlier, e.g., by Alon, Babai, and Itai \cite{ABI}, Luby \cite{Luby}, and Goldberg, Plotkin, and Shannon \cite{GPSh}}. For a comprehensive introduction to this model, see the book \cite{BE} by Barenboim and Elkin. We now briefly describe the key elements of this model.
    
    The \LOCAL model operates on an $n$-vertex graph $G$. \ep{Note that, in contrast to the examples discussed in \S\ref{subsec:Borelgraphs}, the graph $G$ here is finite.} We think of $G$ as representing a decentralized communication network where each vertex plays the role of a processor and edges represent communication links. The computation proceeds in \emph{rounds}. During each round, the vertices first perform some local computations and then synchronously broadcast messages to all their neighbors.  After a number of rounds, every vertex must generate its own part of the output of the algorithm. For example, if the goal of the algorithm is to solve some LCL problem, then each vertex must eventually decide on its own label. The efficiency of such an algorithm is measured by the number of communication rounds required to produce the output. In particular,  there are no restrictions on the complexity of the local computations that the vertices perform during each round \ep{we imagine that the computational resources available to the vertices are infinite} or on the length of the messages that the vertices send to each other.
    
    An important feature of the \LOCAL model is that every vertex of $G$ is executing \emph{the same} algorithm. Therefore, to make this model nontrivial, there must be a way to distinguish the vertices from each other. There are two standard symmetry-breaking approaches, leading to the distinction between deterministic and randomized \LOCAL algorithms:
    \begin{itemize}
        \item In the \emph{deterministic} version of the \LOCAL model, each vertex $x \in V(G)$ is assigned, as part of its input, an identifier $\mathrm{Id}(x)$, which is a string of $\Theta(\log n)$ bits. It is guaranteed that the identifiers assigned to different vertices are distinct. Subject to this guarantee, the algorithm must always output a correct solution to the problem, regardless of the specific assignment of the identifiers.
        
        \item In the \emph{randomized} version of the \LOCAL model, each vertex may generate an arbitrarily long finite sequence of independent uniformly distributed random bits. The algorithm may fail to produce a correct solution to the problem, but the probability of failure must not exceed $1/n$.
    \end{itemize}
    
    We remark that the randomized version of the model is more computationally powerful than the deterministic one. This is because a deterministic \LOCAL algorithm can be simulated by a randomized one: each vertex can simply generate a random sequence of $\lceil 3 \log_2 n\rceil$ bits and use it as its identifier---the probability that two identifiers generated in this way coincide is easily seen to be less than $1/n$.
    
    If $x$ and $y$ are two vertices whose graph distance in $G$ \ep{i.e., the length of the shortest $xy$-path} is $T$, then no information from $y$ can reach $x$ in fewer than $T$ communication rounds \ep{this explains the name ``\LOCAL'' given to the model}. Conversely, in $T$ rounds every vertex can collect all the data present at the vertices at distance at most $T$ from it. Thus, a $T$-round \LOCAL algorithm may be construed simply as a function that, given the structure of the radius-$T$ ball around $x$ \ep{including all the additional information, such as the assignment of the identifiers or random bit sequences to its vertices}, decides on $x$'s output. In other words, the complexity of a \LOCAL algorithm measures how far in the graph an individual vertex should be allowed to ``see'' in order to be able to produce its own output. Naturally, if every vertex is allowed to ``see'' the entire graph, then the model trivializes. Thus, we are interested in algorithms that only allow each vertex to access a relatively small part of the graph.
    
    As an illustration, let us consider a simple example. Suppose that $G$ is an $n$-vertex path. Of course, $\chi(G) = 2$. How fast can a $2$-coloring of $G$ be computed by a \LOCAL algorithm? Obviously, $O(n)$ rounds suffice \ep{because in that many rounds each vertex will have access to the entire graph}. On the other hand, it is fairly easy to show that $o(n)$ rounds are not enough:
    
    \begin{figure}[t]
		\centering	
		\begin{tikzpicture}[xscale=1.5]
		    \node[outer sep=2,inner sep=0] (x) at (1.5,-0.3) {$x$};
		    \node[outer sep=2,inner sep=0] (y) at (4,-0.3) {$y$};
		    \node[outer sep=2,inner sep=0] (z) at (6.7,-0.3) {$z$};
		
		    \fill[orange] (0.8,-0.1) rectangle (2.2,0.1);
			
			\fill[orange] (3.3,-0.1) rectangle (4.7,0.1);
			\fill[orange] (6,-0.1) rectangle (7.4,0.1);
			
			\draw [decorate,decoration={brace,amplitude=6pt}]
			(1.5,0.2) -- (4,0.2) node [black,midway,yshift=0.4cm] {even};
			
			\draw [decorate,decoration={brace,amplitude=30pt}]
			(1.5,0.4) -- (6.7,0.4) node [black,midway,yshift=1.3cm] {odd};
			
			\node (a) at (2.75,-1.5) {same color};
			\draw[-{Stealth[length=1.6mm]}] (a) to (x);
			\draw[-{Stealth[length=1.6mm]}] (a) to (y);
			
			\node (b) at (6,-1.5) {different colors};
			\draw[-{Stealth[length=1.6mm]}] (b) to (x);
			\draw[-{Stealth[length=1.6mm]}] (b) to (z);
			
			\draw [decorate,decoration={brace,amplitude=6pt},color=red]
			(3.3,0.1) -- (4.7,0.1);
			\draw [decorate,decoration={brace,amplitude=6pt},color=red]
			(6,0.1) -- (7.4,0.1);
			
			\node (c) at (6,1.5) {\red{switch!}};
			\draw[-{Stealth[length=1.6mm]},color=red] (c) to (4,0.35);
			\draw[-{Stealth[length=1.6mm]},color=red] (c) to (6.7,0.35);
			
			\draw (0,0) -- (8, 0);
			
			\filldraw (1.5,0) circle (2pt);
			\filldraw (4,0) circle (2pt);
			\filldraw (6.7,0) circle (2pt);
		\end{tikzpicture}
		\caption{An illustration for the proof of Proposition~\ref{prop:path}. The three marked intervals are of length $2T = o(n)$.}\label{fig:path}
	\end{figure}
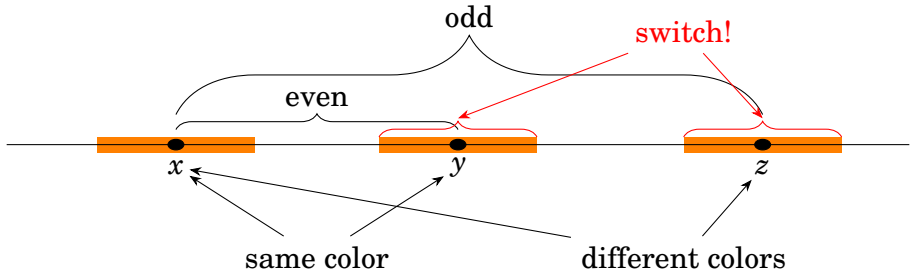
    
    \begin{prop}\label{prop:path}
        There is no $o(n)$-round \LOCAL algorithm \ep{either deterministic or randomized} that finds a $2$-coloring of an $n$-vertex path.
    \end{prop}
    \begin{proof}
        We give an argument in the deterministic case \ep{leaving the randomized case as a nice exercise to the reader}. Suppose that there is a $T$-round deterministic \LOCAL algorithm $\mathcal{A}$ for $2$-coloring an $n$-vertex path, where $T \ll n$. We can pick three vertices $x$, $y$, $z$ such that the distance between $x$ and $y$ is even, while the distance between $x$ and $z$ is odd, and, furthermore, all the pairwise distances between $x$, $y$, and $z$ exceed $2T$ \ep{see Fig.~\ref{fig:path}}. Fix any assignment of identifiers and let $c(x)$, $c(y)$, $c(z)$ be the colors given by the algorithm $\mathcal{A}$ to $x$, $y$, and $z$ respectively. In a $2$-coloring of a path, the two colors alternate, so we must have $c(x) = c(y) \neq c(z)$. Without loss of generality, say $c(x) = c(y) = 1$ and $c(z) = 2$. Now let us exchange the assignments of the identifiers in the interval of length $2T$ centered around $y$ with those in the interval centered around $z$ and let $c'(x)$, $c'(y)$, $c'(z)$ be the colors given to $x$, $y$, and $z$ by the algorithm $\mathcal{A}$ with this new assignment of identifiers. Since the output of $\mathcal{A}$ at a vertex is determined by the identifiers in the interval of length $2T$ centered around it, we have $c'(x) = c'(z) = 1$ and $c'(y) = 2$. This is a contradiction since the colors of $x$ and $y$ must be the same, while the colors of $x$ and $z$ must be different.
    \end{proof}
    
    \subsection{A sample of results}
    
    The first hint of the connection between distributed algorithms and descriptive combinatorics can be obtained by comparing some known results in the two areas. For example, consider the following problem:
    \begin{quote}
        \textsl{Fix a natural number $d \in \N$. What is the minimum $k = k(d)$ such that a $k$-coloring of a graph of maximum degree $d$ can be found efficiently?}
    \end{quote}
    Here the \emph{maximum degree} of a graph is the maximum number of neighbors of a vertex. If we interpret the word ``efficiently'' to mean ``with a fast \LOCAL algorithm,'' then we have the following result:
    
    \begin{theo}[{Goldberg--Plotkin--Shannon \cite[\S3]{BE}}]\label{theo:GPS}
        There exists a deterministic \LOCAL algorithm that finds a $(d+1)$-coloring of an $n$-vertex graph of maximum degree $d$ in $O(\log^\ast n)$ rounds.\footnote{Here and in the sequel we treat $d$ as a constant, meaning that the implied factors in the $O(\cdot)$ notation may depend on $d$.}
    \end{theo}
    
    In the statement of Theorem~\ref{theo:GPS}, $\log^\ast n$ is the \emph{iterated logarithm} of $n$, i.e., the number of times the logarithm function must be iteratively applied to $n$ before the result becomes at most $1$. This is an extremely slow growing function. In particular---and most importantly for us---it is asymptotically of order $o(\log n)$. In a graph of maximum degree $d$, a vertex can ``see'' at most $d^T$ other vertices in $T$ rounds of the \LOCAL model. Therefore, if $T = o(\log n)$, a vertex will definitely not have access to the entire graph. In this sense, $o(\log n)$ rounds is a natural threshold for ``truly local'' algorithms.
    
    Can we reduce the number of colors to $d$? An obvious obstacle is that $G$ may contain a clique on $d + 1$ vertices, all of which would have to receive different colors. On the other hand, the so-called Brooks's theorem in graph theory asserts that if $d \geq 3$ and $G$ is a graph of maximum degree $d$ without a $(d+1)$-clique, then $\chi(G) \leq d$, i.e., a $d$\=/coloring of $G$ exists \cite[Theorem 5.2.4]{Diestel}. Nevertheless, it turns out that no efficient deterministic \LOCAL algorithm can find such a coloring, even if $G$ has no cycles:
    
    \begin{theo}[{Chang--Kopelowitz--Pettie \cite{CKP}}]\label{theo:nod}
        There is no deterministic \LOCAL algorithm that finds a $d$-coloring of an $n$-vertex acyclic graph of maximum degree $d$ in $o(\log n)$ rounds.
    \end{theo}
    
    In contrast to Theorem~\ref{theo:nod}, it is possible to reduce the number of colors using a randomized algorithm:
    
    \begin{theo}[{Ghaffari--Hirvonen--Kuhn--Maus \cite{GHKM}}]\label{theo:yesd}
        There exists a randomized \LOCAL algorithm that finds a $d$-coloring of an $n$-vertex graph $G$ of maximum degree $d$ in $O((\log \log n)^2)$ rounds, provided that $d \geq 3$ and $G$ has no $(d+1)$-cliques.
    \end{theo}
    
    
    Now let us turn to the descriptive combinatorics side of the picture. How many colors do we need for a Borel coloring of a Borel graph of maximum degree $d$? The first result parallels Theorem~\ref{theo:GPS}:
    
    \begin{theo}[{Kechris--Solecki--Todorcevic \cite{KST}}]\label{theo:KST}
        If $G$ is a Borel graph of maximum degree $d$, then $\chi_\mathsf{B}(G) \leq d + 1$.
    \end{theo}
    
    On the other hand, reducing the number of colors to $d$ is impossible even for graphs with no cycles:
    
    \begin{theo}[{Marks \cite{Marks}}]\label{theo:marks}
        For any $d \in \N$, there is an acyclic Borel graph $G$ of maximum degree $d$ such that $\chi_{\mathsf{B}}(G) > d$.
    \end{theo}
    
    However, if we only want a measurable or Baire-measurable coloring, then $d$ colors suffice:
    
    \begin{theo}[{Conley--Marks--Tucker-Drob \cite{CMTD}}]\label{theo:yesdmeas}
        If $G$ is a Borel graph of maximum degree $d$, then both $\chi_\mathsf{M}(G)$ and $\chi_{\mathsf{BM}}(G)$ are at most $d$, provided that $d \geq 3$ and $G$ has no $(d+1)$-cliques.
    \end{theo}
    
    Theorems~\ref{theo:GPS}--\ref{theo:yesd} and \ref{theo:KST}--\ref{theo:yesdmeas} clearly mirror each other. It seems that deterministic \LOCAL algorithms somehow correspond to Borel constructions, while randomized algorithms---to measurable and Baire-measurable ones. It turns out that, indeed, this is not a coincidence and one can prove some general connections of this form.
    
    \subsection{Efficient distributed algorithms yield descriptive results}\label{subsec:dist_to_disc}
    
    In \cite{BerDist}, the following general result is established:
    
    \begin{theo}[{Bernshteyn \cite{BerDist}}]\label{theo:dist_to_desc}
        Let $\mathcal{G}$ be a class of finite graphs closed under adding isolated vertices and let $G$ be a Borel graph of bounded degree all of whose finite induced subgraphs are in $\mathcal{G}$. Fix an LCL problem $\Pi$.
        
        \begin{enumerate}[label=\ep{\itshape{\alph*}}]
            \item\label{item:a} If $\Pi$ can be solved on $n$-vertex graphs from $\mathcal{G}$ by an $o(\log n)$-round deterministic \LOCAL algorithm, then $G$ admits a Borel solution to $\Pi$.
            
            \item\label{item:b} If $\Pi$ can be solved on $n$-vertex graphs from $\mathcal{G}$ by an $o(\log n)$-round randomized \LOCAL algorithm, then $G$ admits both a measurable and a Baire-measurable solution to $\Pi$.
        \end{enumerate}
    \end{theo}
    
    In short, efficient deterministic algorithms yield Borel solutions, while efficient randomized algorithms yield measurable and Baire-measurable solutions. The following implications are special cases of Theorem~\ref{theo:dist_to_desc}:
    \begin{align*}
        \text{Theorem~\ref{theo:GPS}} \ &\Longrightarrow\ \text{Theorem~\ref{theo:KST}},\\
        \text{Theorem~\ref{theo:yesd}} \ &\Longrightarrow\ \text{Theorem~\ref{theo:yesdmeas}},\\
        \text{Theorem~\ref{theo:marks}} \ &\Longrightarrow\ \text{Theorem~\ref{theo:nod}}.
    \end{align*}
    
    In addition to its theoretical significance, Theorem~\ref{theo:dist_to_desc} has a number of applications to specific problems. For instance, if we already know an efficient \LOCAL algorithm for some LCL problem, we can then use it to derive corresponding results in descriptive combinatorics. Several instances of this are given in \cite{BerDist}. For example, Theorem~\ref{theo:dist_to_desc}, together with known distributed algorithms due to Chung, Pettie, and Su \cite{CPS}, is used in \cite[\S3.B]{BerDist} to obtain asymptotically optimal upper bounds on the measurable chromatic number of Borel graphs without short cycles. In the opposite direction, impossibility results in descriptive combinatorics yield lower bounds on the running time of distributed algorithms. In a recent paper \cite{trees}, Brandt \emph{et al.} used this idea to derive new lower bounds for distributed algorithms via a version of the so-called determinacy method that was originally developed by Marks to prove Theorem~\ref{theo:marks}.
    
    Let us now say a few words about the proof of Theorem~\ref{theo:dist_to_desc}. The proof of part \ref{item:a} is actually not too difficult:
    
    \begin{proof}[Proof of Theorem~\ref{theo:dist_to_desc}\ref{item:a} \ep{sketch}]
        Suppose that $\Pi$ is solved on $n$-vertex graphs from $\mathcal{G}$ by an $o(\log n)$-round deterministic \LOCAL algorithm $\mathcal{A}$. The idea is to simulate the execution of the algorithm $\mathcal{A}$ on the given Borel graph $G$ by ``pretending'' that $G$ is finite. 
        
        Let $d$ be the maximum degree of $G$ \ep{which is finite by assumption}. Take a very large natural number $n$ and let $T \ll \log n$ be the running time of $\mathcal{A}$ on $n$-vertex graphs. Let $S$ be the set of all bit strings of length $\lceil \log_2 n\rceil$. 
        
        \begin{claim*}
            There is a Borel labeling $\mathrm{Id} \colon V(G) \to S$ such that $\mathrm{Id}(x) \neq \mathrm{Id}(y)$ whenever $x \neq y$ and the graph distance between $x$ and $y$ is at most $2T+2$.
        \end{claim*}
        \begin{claimproof}[Proof of the claim]
            Define an auxiliary graph $G'$ with the same vertex set as $G$ where distinct vertices $x$ and $y$ are adjacent if and only if their graph distance in $G$ is at most $2T+2$. It is not hard to see that the graph $G'$ is Borel. Since $|S| \geq n$, we just need to find a Borel $n$-coloring of $G'$. The maximum degree of $G'$ is at most $d^{2T+2}$, which is less than $n$ since $2T+2 \ll \log n$. Therefore, by Theorem~\ref{theo:KST}, $G'$ has a Borel $n$-coloring, as desired.
        \end{claimproof}
        
        We treat the bit string $\mathrm{Id}(x)$ as a substitute for an identifier of $x$. By construction, although there may be other vertices with the same identifier, they are far from $x$ in $G$. In particular, if $x$ explores its radius-$T$ ball in $G$, it will only see distinct identifiers. Since the algorithm $\mathcal{A}$ applied on $n$-vertex graphs doesn't allow  a vertex to see outside its radius-$T$ ball, we many run $\mathcal{A}$ on $G$ for $T$ rounds \ep{as if $G$ had $n$ vertices} using the function $\mathrm{Id}$ in place of the identifier assignment. The output labeling will then be a solution to $\Pi$, and, since $\mathrm{Id}$ is a Borel assignment, one can show that the output will also be Borel.
    \end{proof}
    
    Arguments similar to the above proof sketch are common in distributed computing theory. For example, an analogous approach was used by Chang, Kopelowitz, and Pettie in \cite{CKP} to prove that no LCL problem can have deterministic \LOCAL complexity in the interval between $\omega(\log^\ast n)$ and $o(\log n)$.
    
    The proof of Theorem~\ref{theo:dist_to_desc}\ref{item:b} is quite a bit more involved, and we won't venture to explain it here. Let us just mention one interesting feature of it. The key tool used to prove Theorem~\ref{theo:dist_to_desc}\ref{item:b} is the measurable version of the Local Lemma, which is also established in \cite{BerDist}. The Local Lemma is a well-known powerful tool in probabilistic combinatorics that is often used to show that a given LCL problem has a solution. The measurable version of this lemma additionally guarantees that the solution is measurable. And here we encounter another point of interaction between descriptive combinatorics and distributed computing: the proof of the measurable Local Lemma in \cite{BerDist} crucially relies on the efficient randomized \LOCAL algorithm for the Local Lemma developed by Fischer and Ghaffari \cite{FG}.
    
    \subsection{Perfect harmony between the two worlds: continuous solutions}
    
    Theorem~\ref{theo:dist_to_desc} allows one to turn efficient distributed algorithms into results in descriptive combinatorics. It is natural to ask if some sort of converse to Theorem~\ref{theo:dist_to_desc} also holds, i.e., if we can obtain efficient algorithms from descriptive results. While in general this question remains widely open, there is one salient case in which an exact correspondence between the descriptive and the distributed worlds has been established. This case concerns continuous solutions to LCL problems.
    
    \begin{defn}[Continuous solutions to LCL problems]\label{defn:contLCL}
        Let $\Pi$ be an LCL problem with label set $\Lambda$ and let $G$ be a Borel graph. A labeling of the vertices of $G$ by the labels from $\Lambda$ is a \emph{continuous solution} to $\Pi$ if it fulfills all the constraints of the problem $\Pi$ and, additionally, for each label $\lambda \in \Lambda$, the set of all vertices labeled $\lambda$ is open.
    \end{defn}
    
    Note that the set of all vertices receiving a given label in a continuous solution must be both open and closed---\emph{clopen}, in short. Hence, it is only interesting to study continuous solutions to LCL problems on graphs $G$ such that the space $V(G)$ has many clopen subsets. Specifically, we focus on \emph{zero-dimensional} spaces, i.e., spaces whose topology is generated by clopen sets. Even though familiar spaces such as $\R$ do not have any nontrivial clopen sets, zero-dimensional spaces are rather common. For example, the space $\set{0,1}^\N$ is zero-dimensional. Another example of a zero-dimensional Polish space is $\R \setminus \Q$. Also, given any Polish space $X$, it is possible to refine the topology on $X$ to make it zero-dimensional without changing the Borel sets \cite[\S13]{KechrisDST}.
    
    The main result we want to describe in this section says, roughly, that an LCL problem can be solved continuously if and only if it can be solved via an efficient deterministic distributed algorithm. To make this statement precise, we need to specify a particular graph or a class of graphs on which we attempt to solve the problem. While it is possible to make the statement somewhat more general, it will be convenient to consider classes of graphs that look very symmetric, such as regular trees or multi-dimensional grids. The following general definition includes these examples as special cases:
    
    \begin{defn}[$\Gamma$-graphs]
        Let $\Gamma$ be a group generated by a fixed finite symmetric\footnote{A subset $S \subseteq \Gamma$ is \emph{symmetric} if for all $\gamma \in S$, $\gamma^{-1} \in S$ as well.} set $S \subseteq \Gamma$.
        
        The \emph{Cayley graph} $\mathrm{Cay}(\Gamma)$ of $\Gamma$ corresponding to $S$ is the graph with vertex set $\Gamma$ that includes, for every $\gamma \in \Gamma$ and $\sigma \in S$, an edge from $\gamma$ to $\sigma\gamma$. To be precise, we view $\mathrm{Cay}(\Gamma)$ as an edge-labeled graph, where the label of the edge $(\gamma, \sigma \gamma)$ is $\sigma$, but the reader may safely ignore this technicality. 
        
        A \emph{$\Gamma$-graph} is a graph \ep{with edge labels} in which every connected component is isomorphic to $\mathrm{Cay}(\Gamma)$.
    \end{defn}
    
    For instance, the Cayley graph of the additive group $\Z$ with generating set $S = \set{1,-1}$ is a bi-infinite path. Similarly, the Cayley graph of $\Z^d$ with respect to an appropriately chosen generating set is an infinite $d$-dimensional square grid. On the other hand, the Cayley graph of the free group $\F_d$ of rank $d$ is an infinite $d$-regular tree.
    
    Given any group $\Gamma$ with a finite symmetric generating set $S$, there is a canonical way to define a particular $\Gamma$-graph $\mathbf{S}_\Gamma$ on a zero-dimensional vertex set, called the \emph{shift graph} of $\Gamma$. The shift graph $\mathbf{S}_\Gamma$ is an extremely important example that occupies a central place not only in descriptive combinatorics, but also in such areas as ergodic theory and topological dynamics. Before giving the general definition, it will be instructive to consider the special case $\Gamma = \Z$, with generating set $S = \set{1, -1}$. For a subset $A \subseteq \Z$ and an element $n \in \Z$, we write
    \[
        A + n \,\defeq\, \set{m + n \,:\, m \in A}.
    \]
    The set $A + n$ can naturally be seen as a ``shift'' of $A$ by $n$ (hence the term ``shift graph''). We say that a set $A \subseteq \Z$ is \emph{aperiodic} if $A + n \neq A$ for every non-zero $n \in \Z$. This is equivalent to saying that the characteristic function $\mathbf{1}_A$ of $A$ is not periodic, i.e., there is no integer $n \neq 0$ such that $\mathbf{1}_A(m) = \mathbf{1}_A(m + n)$ for all $m \in \Z$. The shift graph $\mathbf{S}_{\Z}$ is the graph whose vertices are the aperiodic subsets of $\Z$ and in which every aperiodic set $A$ has two neighbors: $A + 1$ and $A - 1$. Thus, the connected component of $A$ in $\mathbf{S}_{\Z}$ is the following infinite path:
    \[
      \cdots \ \text{\textemdash}\ (A - 2) \ \text{\textemdash}\  (A - 1) \  \text{\textemdash} \  (A) \  \text{\textemdash}\  (A + 1) \  \text{\textemdash} \ (A + 2) \ \text{\textemdash} \ \cdots 
    \]
    (Since $A$ is aperiodic, all the vertices on this path are distinct.) In particular, every component of $\mathbf{S}_{\Z}$ is isomorphic to the Cayley graph of $\Z$, so it is indeed a $\Z$-graph.
    
    Now for the general construction. Let $\Gamma$ be a group with a finite symmetric generating set $S$. For a set $A \subseteq \Gamma$ and an element $\gamma \in \Gamma$, let
    \[
    \gamma A \,\defeq\, \set{\gamma \alpha \,:\, \alpha \in A}.
    \]
    A subset $A \subseteq \Gamma$ is \emph{aperiodic} if $\gamma A \neq A$ for all non-identity elements $\gamma \in \Gamma$. The vertices of $\mathbf{S}_\Gamma$ are the aperiodic subsets $A \subseteq \Gamma$, and for every $A \in V(\mathbf{S}_\Gamma)$ and $\sigma \in S$, $\mathbf{S}_\Gamma$ includes an edge from $A$ to $\sigma A$ labeled $\sigma$. By construction, for each $A \in V(\mathbf{S}_\Gamma)$, the mapping $\gamma \mapsto \gamma A$ establishes an isomorphism between $\mathrm{Cay}(\Gamma)$ and the connected component of $A$ in $\mathbf{S}_\Gamma$ \ep{the fact that $A$ is aperiodic guarantees that this mapping is injective}. Thus, $\mathbf{S}_\Gamma$ is a $\Gamma$-graph. Furthermore, by identifying each subset of $\Gamma$ with its characteristic function, we can view $V(\mathbf{S}_\Gamma)$ as a subset of $\set{0,1}^\Gamma$, which makes $\mathbf{S}_\Gamma$ a Borel graph on a zero-dimensional Polish space.
    
    We now have the following result, obtained independently by Bernshteyn \cite{Bercont} and Seward (unpublished): 
    
    \begin{theo}[{Bernshteyn \cite{Bercont}/Seward}]\label{theo:cont}
        Let $\Gamma$ be a group generated by a finite symmetric set $S \subseteq \Gamma$. For every LCL problem $\Pi$, the following statements are equivalent:
        
        \begin{enumerate}[label=\ep{\normalfont\roman*}]
            \item\label{item:i} $\mathbf{S}_\Gamma$ admits a continuous solution to $\Pi$.
            
            \item\label{item:ii} There is an $o(\log n)$-round deterministic \LOCAL algorithm that solves $\Pi$ on $n$-vertex subgraphs of the Cayley graph $\mathrm{Cay}(\Gamma)$.
        \end{enumerate}
    \end{theo}
    
    The implication \ref{item:ii} $\Longrightarrow$ \ref{item:i} of Theorem~\ref{theo:cont} is proved in a manner analogous to the proof of Theorem~\ref{theo:dist_to_desc}\ref{item:a} sketched in \S\ref{subsec:dist_to_disc}. The implication \ref{item:i} $\Longrightarrow$ \ref{item:ii} is significantly more difficult and involves a careful analysis of the structure of continuous functions on $V(\mathbf{S}_\Gamma)$. Interestingly, this analysis again relies, among other things, on tools from computer science such as the so-called method of conditional probabilities. It also builds on earlier work of Gao--Jackson--Seward \cite{GJS2}, Seward--Tucker-Drob \cite{STD}, and Elek \cite{Elek}.
    
    Theorem~\ref{theo:cont} explains the analogies between known results in continuous combinatorics and distributed computing. For example, in the case $\Gamma = \Z^d$, the continuous combinatorics of the shift graph have been studied in detail by Gao \emph{et al.}\ \cite{Abelian}. Similarly, Brandt \emph{et al.}\ \cite{grids} investigated the LCL problems that can be solved on $d$-dimensional square grids by efficient distributed algorithms. The results obtained by these two groups of researchers---independently and using different methods---perfectly parallel each other, exactly as Theorem~\ref{theo:cont} predicts.
    
    \subsection{Baire-measurable solutions to LCL problems on trees}
    
    The last result we would like to mention is a very recent theorem that establishes a perfect equivalence between yet another pair of facets of the worlds of descriptive combinatorics and distributed computing: 
    
    \begin{theo}[{Brandt \emph{et al.} \cite{trees}}]\label{theo:Baire}
        Fix $d \in \N$. For every LCL problem $\Pi$, the following statements are equivalent:
        \begin{enumerate}[label=\ep{\normalfont\roman*}]
            \item\label{item:iBaire} Every Borel graph $G$ in which every component is an infinite $d$-regular tree admits a Baire-measurable solution to $\Pi$.
            
            \item\label{item:iiBaire} There is an $O(\log n)$-round deterministic \LOCAL algorithm that solves $\Pi$ on $n$-vertex trees of maximum degree $d$.
        \end{enumerate}
    \end{theo}
    
    This is a truly inspiring result, which highlights how deep the connections between descriptive combinatorics and distributed computing lie. What makes it different from Theorems \ref{theo:dist_to_desc} and \ref{theo:cont} is that the complexity bound on the distributed algorithm here is $O(\log n)$, not $o(\log n)$. Note that in $O(\log n)$ rounds, a vertex of an $n$-vertex tree of maximum degree $d \geq 3$ is guaranteed to discover a vertex of degree less than $d$---and hence this algorithm cannot be directly used on a $d$-regular graph $G$. 
    Unsurprisingly, the proof of the implication \ref{item:iiBaire} $\Longrightarrow$ \ref{item:iBaire} in Theorem~\ref{theo:Baire} is considerably more difficult than the proof of Theorem~\ref{theo:dist_to_desc}\ref{item:a}. In fact, both implications in Theorem~\ref{theo:Baire} are highly nontrivial. The high-level idea of the argument is to isolate a certain combinatorial property of LCL problems and show that this property is equivalent to each of \ref{item:iBaire}, \ref{item:iiBaire} separately \ep{the former equivalence was established by Bernshteyn (unpublished), the latter---by Brandt \emph{et al.} \cite{trees}}.
    
    \subsection{Final thoughts and open problems}
    
    The intimate interactions between descriptive combinatorics and distributed computing have only very recently been discovered, and most questions about them remain open. Perhaps the central open problem is this:
    
    \begin{ques}
        Is there a version of Theorem~\ref{theo:cont} for Borel solutions to LCL problems? In other words, is there a precise algorithmic counterpart to the notion of a Borel solution \ep{similar to how $o(\log n)$-round \LOCAL algorithms are equivalent to continuous solutions}?
    \end{ques}
    
    The same questions for measurable and \ep{except for graphs with no cycles} Baire-measurable solutions are also open. Even on graphs as simple as $d$-regular trees, many open questions remain. The diagram in Fig.~\ref{fig:diagram} summarizes our knowledge regarding the complexity of LCL problems on $d$-regular trees. In particular, we know that there are LCL problems that admit Borel solutions but cannot be solved by an $o(\log n)$-round randomized algorithm, but there are also LCL problems that can be solved by such an algorithm but do not admit Borel solutions. Therefore, a new algorithmic framework seems to be needed to capture the notion of a Borel solution precisely.
    
    \begin{figure*}
		\centering	
		\begin{tikzpicture}
		    \node (detlittle) at (4,2.5) {$\mathsf{DetLOCAL}(o(\log n))$};
		    
		    \node (cont) at (4,0) {$\mathsf{CONTINUOUS}$};
		    
		    \node (Borel) at (7,0) {$\mathsf{BOREL}$};
		    
		    \node (meas) at (10,0) {$\mathsf{MEASURABLE}$};
		    
		    \node (Baire) at (14,0) {$\mathsf{BAIRE\ MEASURABLE}$};
		    
		    \node (Detbig) at (14,2.5) {$\mathsf{DetLOCAL}(O(\log n))$};
		    
		    \node (Ranlittle) at (8.5,2.5) {$\mathsf{RandLOCAL}(o(\log n))$};

		    \draw (detlittle) -- node [midway,above,sloped] {$=$} (cont);
		    \draw (cont) -- node [midway,above,sloped] {$\subset$} (Borel);
		    \draw (Borel) -- node [midway,above,sloped] {$\subset$} (meas);
		    \draw (meas) -- node [midway,above,sloped] {$\subset$} (Baire);
		    \draw (Detbig) -- node [midway,above,sloped] {$=$} (Baire);
		    \draw (Borel) -- node [midway,above,sloped] {$\not\subseteq$} node [midway,below,sloped] {$\not\supseteq$} (Ranlittle);
		    \draw (Ranlittle) -- node [midway,above,sloped] {$\subset$} (meas);
		\end{tikzpicture}
		\caption{Known equalities and inequalities between classes of LCL problems on $d$-regular trees. Here $\mathsf{DetLOCAL}(T(n))$ and $\mathsf{RandLOCAL}(T(n))$ denote the classes of LCL problems solvable by a $T(n)$-round determinisitic (respectively, randomized) \LOCAL algorithm. The symbol ``$\subset$'' indicates strict inclusion.}\label{fig:diagram}
	\end{figure*}
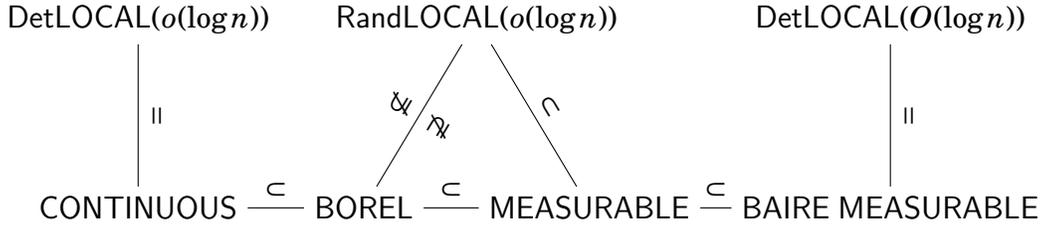
    
    On classes of graphs other than trees, our knowledge is even scanter. For example, the following basic question is open:
    
    \begin{ques}
        Fix an integer $d \geq 2$. Does there exist an LCL problem $\Pi$ such that the shift graph $\mathbf{S}_{\Z^d}$ admits a measurable solution to $\Pi$ but not a Borel one? What about Baire-measurable solutions?
    \end{ques}
    
    Another area with close ties to distributed computing and descriptive combinatorics is the study of so-called \emph{finitary factors of i.i.d.~processes}. This is a class of particularly well-behaved random processes on networks of great interest in probability theory. Finitary factors of i.i.d.~processes can be used to create a complexity hierarchy of LCL problems, and Greb\'ik and Rozho\v{n} recently discovered that in many cases this hierarchy parallels the one arising from considering descriptive and \LOCAL complexity. For more details on this exciting connection, see the papers \cite{GRgrids} by Greb\'ik and Rozho\v{n} and \cite{trees} by Brandt \emph{et al.}
    
    In spite of all the open problems mentioned above, it really does appear that the combination of tools from descriptive set theory and distributed computing is the key to attaining a deep understanding of the behavior of LCL problems under various ``effectiveness'' requirements. Exciting discoveries await, and they will surely enrich descriptive combinatorics and distributed computing alike.
    
    \vspace{10pt}
    
    \noindent \emph{Acknowledgments}. I am very grateful to the anonymous referees for carefully reading this article and providing many helpful suggestions.
    
    \printbibliography
    
\end{document}